\documentclass[11pt,twoside]{article}
\usepackage{amssymb}
\usepackage{amssymb}
\usepackage{amsmath,amsfonts,amscd,amsthm}
\usepackage{color}

\newtheorem{thm}{Theorem}[section]
\newtheorem{prop}[thm]{Proposition}
\newtheorem{lem}[thm]{Lemma}
\newtheorem{cor}[thm]{Corollary}

\theoremstyle{definition}

\theoremstyle{remark}
\newtheorem{rem}[thm]{Remark}
\numberwithin{equation}{section}

\providecommand\ufootnote[1]{{\let\thefootnote\relax\footnote[0]{#1}}}

\newcommand{\cc}{\mathcal C}
\newcommand{\dc}{\mathcal D}

\newcommand{\oc}{\mathcal O}

\newcommand{\cb}{\mathbb C}

\newcommand{\hb}{\mathbb H}
\newcommand{\tb}{\mathbb T}
\newcommand{\ci}{{\mathcal C}^\infty}

\newcommand{\ol}{\overline}

\newcommand{\pa}{\partial}
\newcommand{\opa}{\ol\partial}
\newcommand{\wt}{\widetilde}

\newcommand{\cx}{{\mathbb{C}}}
\newcommand{\dbar}{\overline{\partial}}

\begin{document}

\title{Solving  $\dbar$ with prescribed support on Hartogs triangles  in $\cb^2$ and $\cb\mathbb P^2$}

\author{Christine Laurent-Thi\'{e}baut and Mei-Chi Shaw}

\date{}

% \thanks{Both authors were partially supported by a grant from the  AGIR program of Grenoble INP  and Universit{\'e} Grenoble-Alpes, awarded to Christine Laurent-Thi\'ebaut.  Mei-Chi Shaw is partially supported by an NSF grant.}

\maketitle

\ufootnote{\hskip-0.6cm Universit{\'e} Grenoble-Alpes, Institut Fourier, Grenoble,
F-38041, France\newline CNRS UMR 5582, Institut Fourier,
Saint-Martin d'H\`eres, F-38402, France
\newline Department of Mathematics, University of Notre Dame, Notre Dame, IN 46556, USA}

%\ufootnote{\hskip-0.6cm  {\it A.M.S. Classification}~: .\newline {\it Key words}~: .}

\bibliographystyle{amsplain}
%\setcounter{page}{-1}
%\input resume

%\tableofcontents
%\input dbar01

In this paper we consider  the problem of solving the Cauchy-Riemann equation with prescribed support. More precisely, let $X$ be a complex manifold of complex dimension $n$  and $\Omega\subset X$ a subdomain of $X$. We ask the following questions:

{\sl Let $T$ be a $\opa$-closed $(r,1)$-current, $0\leq r\leq n$, on $X$ with support contained in $\ol\Omega$, does there exist  a $(r,0)$-current on $X$, with support contained in $\ol\Omega$, such that $\opa S=T$ ?

If moreover $T=f$ is a smooth form or a $\cc^k$ form or an $L^p_{loc}$ form, can we find $g$   with support contained in $\ol\Omega$ and with the same regularity  as $f$ such that $\opa g=f$ ?}

This leads us to introduce the Dolbeault cohomology groups with prescribed support in $\overline \Omega$. Let us denote by $H^{r,1}_{\ol\Omega,\infty}(X)$ the quotient space
$$ \{f\in\ci_{r,1}(X)~|~\opa f=0,~{\rm supp}~f\subset\ol\Omega\}/\opa\{f\in\ci_{r,0}(X)~|~{\rm supp}~f\subset\ol\Omega\}.$$
In the same way, we define $H^{r,1}_{\ol\Omega,\cc^k}(X)$, $H^{r,1}_{\ol\Omega,L^p_{loc}}(X)$ and $H^{r,1}_{\ol\Omega,cur}(X)$ for the $\cc^k$, $L^p_{loc}$ and the current category.

The cohomology groups $H^{r,1}_{\ol\Omega,\infty}(X)$, $H^{r,1}_{\ol\Omega,\cc^k}(X)$ ,$H^{r,1}_{\ol\Omega,L^p_{loc}}(X)$ and $H^{r,1}_{\ol\Omega,cur}(X)$ describe the obstruction to solve the Cauchy-Riemann equation with prescribed support in $\ol\Omega$, respectively in the smooth or $\cc^k$ or $L^p_{loc}$ or current category. Their vanishing is equivalent to the solvability of the  Cauchy-Riemann equation with prescribed support in $\ol\Omega$ in the corresponding category (see section 2 in \cite{LaShdualiteL2} and \cite{LaLp}).

Note that, if $\Omega$ is a relatively compact domain with Lipschitz boundary, by the Serre duality, the properties of the groups
 $H^{r,1}_{\ol\Omega,\infty}(X)$,
 %$H^{r,1}_{\ol\Omega,\cc^k}(X)$ , What is the dual of this group? It should be the cohomology of extendable currents of finite order $k$, for $k=0$ these are extendable measures. But we never use this dual in the paper and we don't need to define it.
 $H^{r,1}_{\ol\Omega,L^p_{loc}}(X)$ and $H^{r,1}_{\ol\Omega,cur}(X)$   are directly related to the properties of the Dolbeault cohomology groups $\check{H}^{n-r,n-1}(\Omega)$, $H^{n-r,n-1}_{L^{p'}_{loc}}(\Omega)$, with $1<p<\infty$, $\frac{1}{p}+\frac{1}{p'}=1$ and $H^{n-r,n-1}_\infty(\ol\Omega)$ of Dolbeault cohomology for extendable currents, $L^{p'}$ forms and of smooth forms up to the boundary.

If $\Phi$ is a family of supports in the complex manifold $X$, for example the family, usually denoted by $c$, of all compact subsets of $X$, we can consider the Dolbeault cohomology with support in $\Phi$. The group $H^{r,q}_{\Phi,\infty}(X)$ is the quotient of the space of $\opa$-closed, smooth $(r,q)$-forms on $X$ with support in the family $\Phi$ by the range by $\opa$ of the space of smooth $(r,q-1)$-forms on $X$ with support in the family $\Phi$. Similarly we can define the groups $H^{r,q}_{\Phi,\cc^k}(X)$ ,$H^{r,q}_{\Phi,L^p_{loc}}(X)$ and $H^{r,q}_{\Phi,cur}(X)$. It follows from Corollary 2.15 in \cite{HeLe2} and Proposition 1.2 in \cite{LaLp}, that the Dolbeault isomorphism holds for the Dolbeault cohomology with support condition. This means that all these groups are isomorphic and we   denote them by $H^{p,q}_\Phi(X)$.   In this paper, we will show that such Dolbeault isomorphism no longer holds
if we change   the condition supported in a family of sets in $X$  to prescribed support. 
For   Dolbeault cohomology groups with prescribed support,  the following proposition is proved in Proposition 2.1.

\begin{prop} Let $X$ be a complex manifold and $\Omega\subset X$ a domain in $X$. For any integer $0\leq r\leq {\rm dim}_\cb X$,
the natural morphisms from $H^{r,1}_{\ol\Omega,\infty}(X)$ (resp. $H^{r,1}_{\ol\Omega,\cc^k}(X)$, $k\geq 0$,
 $H^{r,1}_{\ol\Omega,L^p_{loc}}(X)$, $1\leq p\leq +\infty$) into $H^{r,1}_{\ol\Omega,cur}(X)$ are injective. In particular, if  $H^{r,1}_{\ol\Omega,cur}(X)=0$, then $H^{r,1}_{\ol\Omega,\infty}(X)=0$, $H^{r,1}_{\ol\Omega,\cc^k}(X)=0$, $k\geq 0$, and
 $H^{r,1}_{\ol\Omega,L^p_{loc}}(X)=0$.
 \end{prop}

When $\Omega$ is a   Hartogs triangle type set  in $\cb^2$ or   $\cb\mathbb P^2$, we show that the Dolbeault isomorphims fails to hold for the cohomology with prescribed support. When $\Omega$ is an unbounded Hartogs triangle in $\cb^2$, we get

\begin{thm}\label{nonisom}
 If $X=\cb^2$ and $\Omega=\{(z,w)\in\cb^2~|~|z|>|w|\}$, then  $H^{0,1}_{\ol\Omega,\infty}(X)=0$, but $H^{0,1}_{\ol\Omega,\cc^k}(X)\neq 0$, $k\geq 0$, $H^{0,1}_{\ol\Omega,cur}(X)\neq 0$ and  $H^{0,1}_{\ol\Omega,L^2_{loc}}(X)\neq 0$.
  \end{thm}
In the case  when $\Omega$ is a  Hartogs triangle in $\cx\mathbb P^2$, we prove

\begin{thm}\label{cp^2}
 If $X=\cx\mathbb P^2$ and $\Omega=\{[z_0,z_1,z_2]\in\cx\mathbb P^2~|~|z_1|>|z_2|\}$, then  $H^{0,1}_{\ol\Omega,\infty}(X)=0$ and $H^{0,1}_{\ol\Omega,\cc^k}(X)=0$, $k\ge 0$,  but $H^{0,1}_{\ol\Omega,cur}(X)$ and  $H^{0,1}_{\ol\Omega,L^2}(X)\neq 0$ are infinite dimensional and Hausdorff.
 \end{thm}
 The non-vanishing of $H^{0,1}_{\ol\Omega,L^2}(\cx\mathbb P^2)$ is especially interesting since it
   is in sharp contrast to the case of solving $\dbar$ with compact support for a bounded Hartogs triangle in $\cx^2$ (see  Remark 1 at the end of the paper).
 The infinite dimensionality of  $H^{0,1}_{\ol\Omega,L^2}(\cx\mathbb P^2)$ gives the following  result.
 Let  $\dbar_s$ be the  strong $L^2$ closure $\dbar_s: L^2_{2,0}(\Omega)\to L^2_{2,1}(\Omega)$, i.e., the completion of $\dbar$ on  smooth forms up to the boundary in the graph norm.
 Let  $H^{2,1}_{\dbar_s, L^2}(\Omega)$ be the quotient of the kernel of $\dbar_s$ over the range of $\dbar_s$, i.e. the Dolbeault cohomology with respect to the operator $\dbar_s$.
 \begin{cor}\label{dbar_s}    The space   $H^{2,1}_{\dbar_s, L^2}(\Omega)$ is infinite dimensional.
 \end{cor}
 It is not known if $\dbar_s$ agrees with the weak $L^2$ extension or if the range of $\dbar_s$ is closed.
If  the domain $\Omega$  is bounded and Lipschitz, then the weak and strong closure are the same from the Friedrichs' lemma. The Hartogs triangle is a candidate that the weak and strong closure of $\dbar$ might not be the same.

The vanishing of the Dolbeault cohomology groups  with prescribed support in $\ol\Omega$ in bidegree $(0,1)$ is directly related to the extension of holomorphic functions defined on the complement of $\Omega$. This implies the following result:
\begin{prop}
Let $X$ be a complex manifold and $\Omega\subset X$ a domain in $X$. Assume $H^{0,1}_{\ol\Omega,\infty}(X)=0$, then $X\setminus\Omega$ is connected. If moreover $X$ is not compact, $H^{0,1}_c(X)=0$ and $\Omega$ is relatively compact, then $H^{0,1}_{\ol\Omega,\infty}(X)=0$ if and only if $X\setminus\Omega$ is connected.
\end{prop}

We also prove some characterization of pseudoconvexity in terms of Dolbeault cohomology with prescribed support.

\begin{thm}
Let $D$ be a bounded domain in $\cb^2$ with Lipschitz boundary. Then the following assertions are equivalent:

(i) $D$ is a pseudoconvex domain;

(ii) $H^{0,1}_{\ol D,\infty}(\cb^2)=0$ and $H^{0,2}_{\ol D,\infty}(\cb^2)$ is Hausdorff.
\end{thm}

The plan of this paper is as follows: In section 1, we recall some basic  properties of the support and the uniqueness of the solution for $\dbar$. In section 2 we discuss solving $\dbar$ with prescribed support and its relations with the holomorphic extension of functions in various function spaces. In section 3, we study the non-vanishing of Dolbeault cohomology
with prescribed support on the unbounded Hartogs triangle in $\cb^2$.  We analyse the Hartogs triangles in $\cx\mathbb P^2$ in section 4.  Theorems \ref{nonisom} and \ref{cp^2}   provide  interesting examples which give
the non-vanishing for the Dolbeault cohomology groups.  This is in  sharp contrast with the well-known results of solving $\dbar$ for (0,1)-forms  with prescribed support  for bounded domain in $\cx^n$.    We  prove Corollary \ref{dbar_s} using $L^2$  Serre duality. This gives us some insight about the intriguing problem  on weak and strong extension of the $\dbar$ operator in the $L^2$ sense, when the domain is not Lipschitz. The unbounded Hartogs domain in $\cb^2$  or Hartogs domains in $\cx\mathbb P^2$  provide us with   new     unexpected phenomena.   Many   open questions   and remarks   are given at the end of the paper.

\bigskip
\noindent
{\em Acknowledgements:} Mei-Chi Shaw would like to thank Phil Harrington  for helpful discussions on  the extension of functions from the Hartogs triangle.

Both authors were partially supported by a grant from the  AGIR program of Grenoble INP  and Universit{\'e} Grenoble-Alpes, awarded to Christine Laurent-Thi\'ebaut.  Mei-Chi Shaw is partially supported by an NSF grant.

\section{Properties of  the support and uniqueness of the solution}

Let $X$ be a complex manifold of complex dimension $n$ and $T$ be a $\opa$-exact $(0,1)$-current on $X$. We will describe some relations between the support of the current $T$ and the support of  the  solution $S$  of the Cauchy-Riemann equation $\opa S=T$.

\begin{prop}\label{support}
Let $X$ be a complex manifold of complex dimension $n$ and $T$ be a $\opa$-exact $(0,1)$-current on $X$. If $\Omega^c$ denotes a connected component of $X\setminus{\rm supp}~T$ and if $S$ is a distribution on $X$ such that $\opa S=T$, then either ${\rm supp}~S\cap\Omega^c=\emptyset$ or $\Omega^c\subset{\rm supp}~S$.
\end{prop}
\begin{proof}
Note that, since $\opa S=T$, $S$ is a holomorphic function on $X\setminus{\rm supp}~T$ and in particular on the connected set $\Omega^c$. Assume that the support of $S$ does not contain $\Omega^c$, then $S$ vanishes on an open subset of $\Omega^c$ and by analytic continuation $S$ vanishes on $\Omega^c$, which means that ${\rm supp}~S\cap\Omega^c=\emptyset$.
\end{proof}

\begin{cor}
Let $X$ be a complex manifold of complex dimension $n$ and $T$ be a $\opa$-exact $(0,1)$-current on $X$. Assume that $X\setminus{\rm supp}~T$ is connected, then if $S$ is a distribution on $X$ such that $\opa S=T$, then either ${\rm supp}~S={\rm supp}~T$ or ${\rm supp}~S=X$.
\end{cor}
\begin{proof}
The support of $T$ is always contained in the support of $S$. If ${\rm supp}~S\neq X$, then the other inclusion holds by  Proposition 1.1  since $X\setminus{\rm supp}~T$ is connected.
\end{proof}

Note that the difference between two solutions of the equation $\opa S=T$ is a holomorphic function on $X$. Then analytic continuation implies the following  uniqueness result.

\begin{prop}\label{uniqueness}
Assume that the complex manifold $X$ is connected. Let  $T$ be a $\opa$-exact $(0,1)$-current on $X$ such that $X\setminus{\rm supp}~T\neq\emptyset$ and $S$ and $U$ two distributions such that
$$\opa S=\opa U=T$$
and the support of $S$ and the support of $U$ do not intersect on  the same connected component $\Omega^c$ of $X\setminus{\rm supp}~T$, then $S=U$.

In particular, the equation $\opa S=T$ admits at most one solution $S$ such that ${\rm supp}~S={\rm supp}~T$.
\end{prop}

\begin{rem}
The equation $\opa S=T$ may have no solution $S$ with ${\rm supp}~S={\rm supp}~T$. Consider for example a relatively compact domain $D$ with $\ci$-smooth boundary in a complex manifold $X$ and a function $F\in\ci(\ol D)$ which is holomorphic in $D$. Denote by $f$ the restriction of $F$ to the boundary of $D$ and set $S=F\chi_D$, where $\chi_D$ is the characteristic function of the domain $D$. Then, by the Stokes formula, $\opa S=f[\pa D]^{0,1}$, where $[\pa D]^{0,1}$ is the part of bidegree $(0,1)$ of the integration current over the boundary of $D$. Clearly the support of $T=f[\pa D]^{0,1}$ is the boundary of $D$, but, by Proposition \ref{uniqueness}, $S$ is the unique solution of $\opa S=T$ whose support is contained in $\ol D$. So there is no solution whose support is equal  to the support of $T$.
\end{rem}

Let us end this section by considering the regularity of the solutions.
\begin{prop}\label{reg}
Let $X$ be a complex manifold and $f$ a $(0,1)$-form with coefficients in $\cc^k(X)$, $0\leq k\leq +\infty$ (resp. $L^p_{loc}(X)$, $1\leq p\leq +\infty$), which is $\opa$-exact in the sense of currents. Then any solution $g$ of the equation $\opa g=f$ is in $\cc^k(X)$, $0\leq k\leq +\infty$ (resp. $L^p_{loc}(X)$, $1\leq p\leq +\infty$).
\end{prop}
\begin{proof}
By the regularity of the Cauchy-Riemann operator (injectivity of the Dolbeault isomorphism \cite{HeLe2} and \cite{LaLp}), if $f$ has coefficients in $\cc^k(X)$, $0\leq k\leq +\infty$ (resp. $L^p_{loc}(X)$, $1\leq p\leq +\infty$), then, since $f$ is $\opa$-exact in the sense of currents, the equation $\opa S=f$ has a solution in $\cc^k(X)$, $0\leq k\leq +\infty$ (resp. $L^p_{loc}(X)$, $1\leq p\leq +\infty$). The difference between two solutions of the equation $\opa S=f$ being a holomorphic function on $X$, all the solutions have the same regularity.
\end{proof}

Associating Proposition \ref{uniqueness} and Proposition \ref{reg}, we get:
\begin{cor}
Assume that the complex manifold $X$ is connected. If $f$ is a $(0,1)$-form such that $X\setminus {\rm supp}~f\neq\emptyset$, then the equation $\opa g=f$ has at most one unique solution such that ${\rm supp}~g={\rm supp}~f$ and this solution has the same regularity as $f$.
\end{cor}

\section{Solving $\dbar$  with prescribed support}

Let $X$ be a connected, complex manifold and $\Omega$ a domain such that $\ol\Omega$ is strictly contained in $X$ and the interior of $\ol\Omega$ coincides with $\Omega$. We set $\Omega^c=X\setminus \ol\Omega$, it is a non-empty open subset of $X$.

Let us denote by $H^{0,1}_{\ol\Omega,\infty}(X)$ (resp.
$H^{0,1}_{\ol\Omega,cur}(X)$, $H^{0,1}_{\ol\Omega,\cc^k}(X)$, $H^{0,1}_{\ol\Omega,L^p_{loc}}(X)$) the Dolbeault
cohomology group of bidegree $(0,1)$ for smooth forms (resp.
currents, $\cc^k$-forms, $k\geq 0$, $L^p_{loc}$-forms, $1\leq p\leq +\infty$) with support in $\ol\Omega$.
The vanishing of these groups means that one can solve the $\opa$
equation with prescribed support in $\ol\Omega$ in the smooth category
(resp. the space of currents, the space of $\cc^k$-forms, the space of $L^p_{loc}$-forms).

It follows from Proposition \ref{uniqueness}, Proposition \ref{reg} and from the Dolbeault isomorphism with support conditions (Corollary 2.15 in \cite{HeLe2} and Proposition 1.2 in \cite{LaLp}) that
\begin{prop}\label{injective}
The natural morphisms from $H^{0,1}_{\ol\Omega,\infty}(X)$ (resp. $H^{0,1}_{\ol\Omega,\cc^k}(X)$, $k\geq 0$,
 $H^{0,1}_{\ol\Omega,L^p_{loc}}(X)$, $1\leq p\leq +\infty$) into $H^{0,1}_{\ol\Omega,cur}(X)$ are injective. In particular, if  $H^{0,1}_{\ol\Omega,cur}(X)=0$, then $H^{0,1}_{\ol\Omega,\infty}(X)=0$, $H^{0,1}_{\ol\Omega,\cc^k}(X)=0$ and
 $H^{0,1}_{\ol\Omega,L^p_{loc}}(X)=0$.
 \end{prop}

In the next sections, examples are given proving that there exist domains in $\cb^2$ and $\cb P^2$ such that $H^{0,1}_{\ol\Omega,\infty}(X)=0$, but $H^{0,1}_{\ol\Omega,cur}(X)\neq 0$.

 We will now consider the link between the vanishing of the group $H^{0,1}_{\ol\Omega,cur}(X)$ and the extension properties of some holomorphic functions in $\Omega^c$.

\begin{prop}\label{CNcur}
Assume $H^{0,1}_{\ol\Omega,cur}(X)=0$, then any holomorphic function on $\Omega^c=X\setminus \ol\Omega$, which is the restriction to $\Omega^c$ of a distribution on $X$, extends as a holomorphic function to $X$.
\end{prop}
\begin{proof}
Let $f\in \oc(\Omega^c)$ and $S_f\in\dc'(X)$ a  distribution such that ${S_f}_{|_{\Omega^c}}=f$. Consider the $(0,1)$-current $\opa S_f$, it is closed and has support in $\ol\Omega$. Since $H^{0,1}_{\ol\Omega,cur}(X)=0$, there exists $U\in\dc'(X)$, with support in $\ol\Omega$ such that $\opa U=\opa S_f$ in $X$. Set $h=S_f-U$, it is a holomorphic fonction on $X$ and $h_{|_{\Omega^c}}=S_{f|_{\Omega^c}}=f$.
\end{proof}

In the same way, we can prove
\begin{prop}\label{CNLp}
Assume $H^{0,1}_{\ol\Omega,L^p_{loc}}(X)=0$, $p\geq 1$, then any holomorphic function on $\Omega^c=X\setminus \ol\Omega$, which is the restriction to $\Omega^c$ of a form with coefficients in $W^{1,p}_{loc}(X)$, extends as a holomorphic function to $X$.
\end{prop}

\begin{prop}\label{CNCk}
Assume $H^{0,1}_{\ol\Omega,\cc^k}(X)=0$, $k\geq 0$, then any holomorphic function on $\Omega^c=X\setminus \ol\Omega$, which is of class $\cc^{k+1}$ on $X\setminus\Omega=\ol{\Omega^c}$, extends as a holomorphic function to $X$.
\end{prop}

\begin{prop}\label{CNsmooth}
Assume $H^{0,1}_{\ol\Omega,\infty}(X)=0$, then any holomorphic function on $\Omega^c=X\setminus \ol\Omega$, which is smooth on $X\setminus\Omega=\ol{\Omega^c}$, extends as a holomorphic function to $X$.
\end{prop}

\begin{cor}\label{CSconnected}
Assume $H^{0,1}_{\ol\Omega,\infty}(X)=0$, then $\Omega^c=X\setminus \ol\Omega$ is connected.
\end{cor}
\begin{proof}
Assume $\Omega^c$ is not connected. Let $f$ be a holomorphic function which is constant equal to $1$ in one connected component of $\Omega^c$ and vanishes identically on all the other ones. By analytic continuation $f$ cannot be the restriction to $\Omega^c$ of a holomorphic function on $X$, and by Proposition \ref{CNsmooth} we get $H^{0,1}_{\ol\Omega,\infty}(X)\neq 0$.
\end{proof}

\begin{rem}
Note that, by Proposition \ref{support}, $H^{0,1}_{\ol\Omega,cur}(X)\neq 0$ if and only if there exists at least one $\opa$-exact $(0,1)$-current $T$ with support contained in $\ol\Omega$ such that the support of each solution of the equation $\opa S=T$ contains at least a connected component of $\Omega^c$.
\end{rem}

Let us give a partial converse to Corollary \ref{CSconnected}. Let $H^{0,1}_c(X) $ denote the Dolbeault cohomology group for $(0,1)$-forms  with
compact support in $X$.

\begin{prop}\label{CNconnected}
Assume $\Omega$ is relatively compact in a non-compact complex manifold $X$ such that $H^{0,1}_c(X)=0$. If $\Omega^c=X\setminus \ol\Omega$ is connected, then
$$H^{0,1}_{\ol\Omega,cur}(X)=H^{0,1}_{\ol\Omega,\infty}(X)=H^{0,1}_{\ol\Omega,\cc^k}(X)=H^{0,1}_{\ol\Omega,L^p_{loc}}(X)=0.$$
\end{prop}
\begin{proof}
By Proposition \ref{reg}, it suffices  to prove that $H^{0,1}_{\ol\Omega,cur}(X)=0$. This vanishing result follows directly from Proposition \ref{support}. More precisely, if $T$ is a $\opa$-closed current on $X$ with support contained in $\ol\Omega$, there exists a distribution $S$, with compact support  such that $\opa S=T$, since $H^{0,1}_c(X)=0$. Then the support of $S$ cannot contain the connected set $\Omega^c$, otherwise $X=\ol\Omega\cup {\rm supp}~S$ would be compact, and hence ${\rm supp}~S$ is contained in $\ol\Omega$.
\end{proof}

In particular, if $X$ is a Stein manifold with ${\rm dim}_\cb~X\geq 2$ and $\Omega$ a relatively compact domain in $X$, then

\centerline{$H^{0,1}_{\ol\Omega,cur}(X)=H^{0,1}_{\ol\Omega,\infty}(X)=H^{0,1}_{\ol\Omega,\cc^k}(X)=H^{0,1}_{\ol\Omega,L^p_{loc}}(X)=0$~$\Leftrightarrow$~$\Omega^c$ is connected.}
\medskip
An immediate corollary of Proposition \ref{CNconnected} and Proposition \ref{CNcur} is the following:

\begin{cor}\label{extension-holo}
Let $X$ be a non-compact, connected complex manifold such that  $H^{0,1}_c(X)=0$, and $\Omega$ a relatively compact, open subset of $X$ with connected complement, then any holomorphic function on $\Omega^c$ extends as a holomorphic function to $X$.
\end{cor}
\begin{proof}
It is sufficient to apply Proposition \ref{CNconnected} and Proposition \ref{CNcur} to a neighborhood $D$ of $\ol\Omega$ with connected complement and to conclude by analytic continuation.
\end{proof}

Corollary \ref{extension-holo} is   the classical Hartogs extension phenomenon.
Note that all the previous results remain true if we replace the family of all compact subsets of a non-compact manifold by any family $\Phi$ of supports in a manifold $X$, different from the family of all closed subsets of $X$ (see e.g. \cite{Se} for the definition of a family of supports).

\begin{prop}\label{CSsmooth}
Assume the complex manifold $X$ satisfies $H^{0,1}(X)=0$. If any holomorphic function on $\Omega^c$, which is smooth on $X\setminus\Omega=\ol{\Omega^c}$, extends as a holomorphic function to $X$, then $H^{0,1}_{\ol\Omega,\infty}(X)=0$.
\end{prop}
\begin{proof}
Let $f$ be a smooth
$\opa$-closed form in $X$ with support contained in
$\ol\Omega$. Since $H^{0,1}(X)=0$, there exists a function
$g\in\ci(X)$ such that $\opa g=f$. Since the support of $f$ is contained in
$\ol\Omega$, $g$ is holomorphic in $\Omega^c$ and by the extension
property it extends as a holomorphic function $\wt g$ to $X$. Set
$h=g-\wt g$, then the support of $h$ is contained in
$\ol\Omega$ and $\opa h=f$.
\end{proof}

 Similarly,   since $H^{0,1}(X)=H^{0,1}_{\cc^k}(X)=H^{0,1}_{L^p_{loc}}(X)=H^{0,1}_{cur}(X)=0$ by the Dolbeault isomorphism,  we have
\begin{prop}\label{CSCk}
Assume the complex manifold $X$ satisfies $H^{0,1}(X)=0$. If any holomorphic function on $\Omega^c$, which is of class $\cc^{k}$, $k\geq 0$ on $X\setminus\Omega=\ol{\Omega^c}$, extends as a holomorphic function to $X$, then $H^{0,1}_{\ol\Omega,\cc^k}(X)=0$.
\end{prop}

\begin{prop}\label{CSLp}
Assume the complex manifold $X$ satisfies $H^{0,1}(X)=0$. If any holomorphic function on $\Omega^c=X\setminus \ol\Omega$, which is the restriction to $\Omega^c$ of a  function    $L^p_{loc}(X)$, $p\geq 1$, extends as a holomorphic function to $X$, then $H^{0,1}_{\ol\Omega,L^p_{loc}}(X)=0$.
\end{prop}

\begin{prop}\label{CScur}
Assume the complex manifold $X$ satisfies $H^{0,1}(X)=0$. If any holomorphic function on $\Omega^c=X\setminus \ol\Omega$, which is the restriction to $\Omega^c$ of a distribution on $X$, extends as a holomorphic function to $X$, then $H^{0,1}_{\ol\Omega,cur}(X)=0$.
\end{prop}

Let us end this section by a characterization  of pseudoconvexity in $\cb^2$  by means of the Dolbeault cohomology with prescribed support.

\begin{thm}
Let $D$ be a bounded domain in $\cb^2$ with Lipschitz boundary. Then the following assertions are equivalent:

 (\text{i})  $D$ is a pseudoconvex domain;

(\text{ii}) $H^{0,1}_{\ol D,\infty}(\cb^2)=0$ and $H^{0,2}_{\ol D,\infty}(\cb^2)$ is Hausdorff.

\end{thm}
\begin{proof}
 By Serre duality (\cite{Ca} or Theorem 2.7 in \cite{LaShdualiteL2}) assertion (ii) implies that
 $\check{H}^{2,q}(D)$ is Hausdorff, for all $1\leq q\leq 2$ and moreover  $\check{H}^{2,1}(D)=0$ as the dual space to $H^{0,1}_{\ol D,\infty}(\cb^2)$.
 Let us prove now that the condition $\check{H}^{2,1}(D)=0$ implies that $D$ is pseudoconvex.
 We will follow the methods used by Laufer \cite{Lf} for the usual Dolbeault cohomology and prove  by contradiction.

Assume $D$ is not pseudoconvex, then there exists a domain $\wt D$ strictly containing $D$ such that any holomorphic function on $D$ extends holomorphically to $\wt D$. Since interior($\ol D$)$=D$, after a translation and a rotation we may assume that $0\in\wt D\setminus\ol D$ and there exists a point $z_0$ in the intersection of the plane $\{(z_1,z_2)\in\cb^2~|~z_1=0\}$ with $D$, which belongs to the same connected component of the intersection of that plane with $\wt D$.

Let us denote by $B(z_1,z_2)$ the $(0,1)$-form defined by
$$B(z_1,z_2)=\frac{\ol z_1~d\ol z_2-\ol z_2~d\ol z_1}{|z|^4}\wedge dz_1\wedge dz_2.$$
It is derived from the Bochner-Martinelli kernel in $\cb^2$ and   is a $\opa$-closed form on $\cb^2\setminus\{0\}$. Then the $L^1$-form $\frac{\ol z_2}{|z|^2}\wedge dz_1\wedge dz_2$ defines a distribution in $\cb^2$ which satisfies
$$\opa(\frac{-\ol z_2}{|z|^2}~dz_1\wedge dz_2)=z_1B(z_1,z_2)\qquad \text{ on }\cb^2\setminus\{0\}.$$
On the other hand, if  $\check{H}^{2,1}(D)=0$, there exists an extendable $(2,0)$-current $v$ such that $\opa v=B$ on $D$ and by the  regularity of   $\opa$ in bidegree $(2,1)$, $v$ is smooth on $D$, since $B$ is smooth on $\cb^2\setminus\{0\}$. Set
$$F=z_1v+\frac{\ol z_2}{|z|^2}\wedge dz_1\wedge dz_2.$$ Then $F$ is a holomorphic $(2,0)$-form on $D$, so its coefficient $F_{12}$
should extend holomorphically to $\wt D$, but we have $F_{12}(0,z_2)=\frac{1}{z_2}$ on $D\cap\{z_1=0\}$, which is holomorphic and singular at $z_2=0$.  This    gives the contradiction since $0\in\wt D\setminus D$.
This proves that  (ii) $\Rightarrow$ (i).

For the converse, first note that if $D$ is a pseudoconvex domain in $\cb^2$, then $\cb^2\setminus D$ is connected and by Proposition \ref{CNconnected}, we have $H^{0,1}_{\ol D,\infty}(\cb^2)=0$. Then we apply Theorem 5 in \cite{CS2012} to get that if $D$ is pseudoconvex with Lipschitz boundary,  then $H^{0,1}_\infty(\cb^2\setminus D)$ is Hausdorff. Let us prove that if $H^{0,1}_\infty(\cb^2\setminus D)$ is Hausdorff,  then $H^{0,2}_{\ol D,\infty}(\cb^2)$ is Hausdorff.

Let $f$ be a $\opa$-closed $(0,2)$-form on $\cb^2$ with support contained in $\ol D$ such that for any  $\opa$-closed $(2,0)$-current $T$ on $D$ extendable as a current to $\cb^2$, we have $<T,f>=0$. Since $H^{0,2}(\cb^2)=0$, there exists a smooth $(0,1)$-form $g$ on $\cb^2$ such that $\opa g=f$ on $\cb^2$, in particular $\opa g=0$ on $\cb^2\setminus\ol D$.

Let $S$ be any $\opa$-closed $(2,1)$-current on $\cb^2$ with compact support in $\cb^2\setminus D$, then, since $H^{2,1}_c(\cb^2)=0$, there exists a compactly supported  $(2,0)$-current $U$ on $\cb^2$ such that $\opa U=S$ and in particular $\opa U=0$ on $D$.

 Thus
$$<S,g>=<\opa U,g>=<U,\opa g>=<U,f>=0,$$
by hypothesis on $f$.
Therefore the Hausdorff property of $H^{0,1}_{\infty}(X\setminus D)$ implies there exists a smooth function $h$ on $X\setminus D$ such that $\opa h=g$. Let $\wt h$ be a smooth extension of $h$ to $\cb^2$, then $u=g-\opa\wt h$ is a smooth form with support in $\ol D$ and
$$\opa u=\opa (g-\opa\wt h)=\opa g=f.$$
This proves that $H^{0,2}_{\ol D,\infty}(\cb^2)$ is Hausdorff,  which  proves that  (i) $\Rightarrow$ (ii).
\end{proof}

\section{The case of the unbounded Hartogs triangle in $\cb^2$}

In $\cb^2$, let us define the domains $\hb^+$ and $\hb^-$ by
\begin{align*}
\hb^+&=\{(z,w)\in\cb^2~|~|z|<|w|\}\\
\hb^-&=\{(z,w)\in\cb^2~|~|z|>|w|\}
\end{align*}
then $\hb^+\cap\hb^-=\emptyset$ and $\ol\hb^+\cup\ol\hb^-=\cb^2$.

Let us denote by $H^{0,1}_{\ol\hb^-,\infty}(\cb^2)$ (resp.
$H^{0,1}_{\ol\hb^-,cur}(\cb^2)$, $H^{0,1}_{\ol\hb^-,L^2_{loc}}(\cb^2)$,$H^{0,1}_{\ol\hb^-,\cc^k}(\cb^2)$ ) the Dolbeault
cohomology group of bidegree $(0,1)$ for smooth forms (resp.
currents, $L^2_{loc}$-forms, $\cc^k$-forms) with support in $\ol\hb^-$.

The vanishing of these groups means that one can solve the $\opa$
equation with prescribed support in $\ol\hb^-$ in the smooth category
(resp. the space of currents, the space of $L^2$-forms, the space of $\cc^k$-forms).
\medskip

We can apply Propositions \ref{CNsmooth} and \ref{CSsmooth} for $\Omega=\hb^-$, since $H^{0,1}(\cb^2)=0$, and we get

\begin{prop}\label{smooth}
We have $H^{0,1}_{\ol\hb^-,\infty}(\cb^2)=0$ if and only if any
holomorphic function on $\hb^+$ which is smooth on $\ol\hb^+$ extends
as a holomorphic function to $\cb^2$.
\end{prop}

\begin{prop}\label{ext}
Any holomorphic function on $\hb^+$ which is smooth on $\ol\hb^+$ extends
as a holomorphic function to $\cb^2$.
\end{prop}
\begin{proof}
Let $f\in\ci(\ol\hb^+)\cap \oc(\hb^+)$. By Sibony's result (\cite{Si}, page 220), for any
$R>0$, the restriction of $f$ to
$\hb^+\cap\Delta(0,R)\times\Delta(0,R)$ extends holomorphically to the
bidisc $\Delta(0,R)\times\Delta(0,R)$ and then by analytic continuation $f$
extends holomorphically to $\cb^2$.
\end{proof}

It follows immediately from Proposition \ref{smooth} and Proposition
\ref{ext} that
\begin{cor}\label{vanishsmooth}
$H^{0,1}_{\ol\hb^-,\infty}(\cb^2)=0$.
\end{cor}

Let us consider now the case of currents. We can apply Proposition \ref{CNCk} to get
\begin{prop}\label{cur}
Assume we have $H^{0,1}_{\ol\hb^-,\cc^k}(\cb^2)=0$, $k\geq 0$ then any
holomorphic function on $\hb^+$, which is of class $\cc^{k+1}$ on $\ol\hb^+$, extends
as a holomorphic function to $\cb^2$.
\end{prop}

\begin{thm}\label{ck}
For any $k\geq 0$, $H^{0,1}_{\ol\hb^-,\cc^k}(\cb^2)\neq 0$,  and $H^{0,1}_{\ol\hb^-,cur}(\cb^2)\neq 0$.
\end{thm}
\begin{proof}
Let us consider the function $h$ define on $\hb^+$ by $h(z,w)=z^l
(\frac{z}{w})$, $l\geq 0$. It is of class $\cc^{k+1}$ on $\ol\hb^+$, if
$l\geq k+2$, but does not extend as a  holomorphic function to $\cb^2$. In
fact if $h$   admits a  holomorphic extension $\wt h$ to $\cb^2$,
then we would have
$$\wt h(z,w)=z^l(\frac{z}{w})\quad {\rm on}\quad \cb^2\setminus\{w=0\},$$
which is not bounded nearby $\{(z,w)\in\cb^2~|~z\neq 0, w=0\}$.
By Proposition \ref{cur}, we get $H^{0,1}_{\ol\hb^-,\cc^k}(\cb^2)\neq 0$.
Then using Proposition \ref{injective} , it follows  $H^{0,1}_{\ol\hb^-,cur}(\cb^2)\neq 0$.

\end{proof}

 Proposition \ref{smooth} still holds if we replace
smooth forms by $W^1_{loc}$-forms (for $D\subset\cb^2$, $W^1_{loc}(\ol D)$
is the space of functions which are in $W^1(\ol D\cap B(0,R))$ for
any $R>0$) in the
following way
\begin{prop}\label{L2}
We have $H^{0,1}_{\ol\hb^-,L^2_{loc}}(\cb^2)=0$ if and only if any
 function $f\in\oc(\hb^+)\cap W^1_{loc}(\ol\hb^+)$,  which is the restriction to $\ol\hb^+$ of a  form with coefficients in $W^1_{loc}(\cb^2)$, extends as a
 holomorphic function to $\cb^2$.
\end{prop}

\begin{thm}
$H^{0,1}_{\ol\hb^-,L^2_{loc}}(\cb^2)\neq 0$
\end{thm}
\begin{proof}
Let us consider the function $h$ defined on $\hb^+$ by $h(z,w)=z^3
(\frac{z}{w})$. It is of class $\cc^2$ on $\ol\hb^+$ and it is in
$W^{1}_{loc}(\ol\hb^+)$ and extends as a $\cc^2$ funtion to $\cb^2$ by
the Whitney extension Theorem, but does not extend as a  holomorphic
function to $\cb^2$. In
fact if $h$ would admit a  holomorphic extension $\wt h$ to $\cb^2$,
then we would have
$$\wt h=z^3(\frac{z}{w})\quad {\rm on}\quad \cb^2\setminus\{w=0\},$$
which is not bounded nearby $\{(z,w)\in\cb^2~|~z\neq 0, w=0\}$.
By Proposition \ref{L2}, we get $H^{0,1}_{\ol\hb^-,L^2_{loc}}(\cb^2)\neq 0$.

\end{proof}

{\parindent=0pt{\bf Remark}: Note that if we replace $\hb^-$ by the classical Hartogs triangle $\tb^-=\hb^-\cap\Delta\times\Delta$, where $\Delta$ is the unit disc in $\cb$, then by Proposition \ref{CNconnected} we have $$H^{0,1}_{\ol\tb^-,L^2_{loc}}(\cb^2)=H^{0,1}_{\ol\tb^-,L^2_{loc}}(\cb^2)=H^{0,1}_{\ol\tb^-,\infty}(\cb^2)=0.$$
So for solving the $\opa$-equation with prescribed support, it is quite different to consider a bounded domain or an unbounded domain as support.}

\section{The case of the Hartogs triangles in $\cx\mathbb P^2$}

In $\cx\mathbb P^2$,  we denote the homogeneous coordinates by
$[z_0,z_1,z_2]$. On the domain where $z_0\neq 0$, we set
$z=\frac{z_1}{z_0}$ and $w=\frac{z_2}{z_0}$.
Let us define the domains $\hb^+$ and $\hb^-$ by
\begin{align*}
\hb^+&=\{[z_0:z_1:z_2]\in\cx\mathbb P^2~|~|z_1|<|z_2|\}\\
\hb^-&=\{[z_0:z_1:z_2]\in\cx\mathbb P^2~|~|z_1|>|z_2|\}
\end{align*}
then $\hb^+\cap\hb^-=\emptyset$ and $\ol\hb^+\cup\ol\hb^-=\cx\mathbb P^2$. These domains are called Hartogs' triangles in $\cx\mathbb P^2$.
The Hartogs triangles  provide examples of non-Lipschitz Levi-flat hypersurfaces  (see \cite{HI}).

  For $k\geq 0$ or $k=\infty$, we  denote by $H^{0,1}_{\ol\hb^-,\cc^k}(\cx\mathbb P^2)$  (resp.
$H^{0,1}_{\ol\hb^-,cur}(\cx\mathbb P^2)$, $H^{0,1}_{\ol\hb^-,L^2}(\cx\mathbb P^2)$) the Dolbeault
cohomology group of bidegree $(0,1)$ for $\cc^k$-smooth forms  (resp.
currents, $L^2$-forms) with support in $\ol\hb^-$.

Again the vanishing of these groups means that one can solve the $\opa$
equation with prescribed support in $\ol\hb^-$ in the $\cc^k$-smooth category
(resp. the space of currents, the space of $L^2$-forms).
\medskip

We can also apply Propositions \ref{CNsmooth} and \ref{CSsmooth} for $\Omega=\hb^-$, since $H^{0,1}(\cx\mathbb P^2)=0$, and we get

 \begin{prop}
We have, for $k\geq 0$ and for $k=\infty$,  $H^{0,1}_{\ol\hb^-,\cc^k}(\cx\mathbb P^2)=0$ if and only if any
holomorphic function on $\hb^+$ which is $\cc^{k+1}$-smooth on $\ol\hb^+$ extends
as a  holomorphic function to $\cx\mathbb P^2$.
\end{prop}

\begin{prop}
Any holomorphic function on $\hb^+$ which is continuous  on $\ol\hb^+$ is constant.
\end{prop}
\begin{proof}
Let $f\in \cc(\ol\hb^+)\cap \oc(\hb^+)$. Notice that the boundary $b \hb^+$  of  $\hb^+$ is foliated by a family of compact complex curves
described in non-homogeneous  coordinates by
\begin{equation} S_\theta =\{ z=e^{i\theta} w\},\qquad \theta\in \mathbb R.\end{equation}
 Restricted to each fixed $\theta$,  $f$ is a continuous $CR$ function on the  compact Riemann surface $S_\theta$.
 Thus $f$ must be a constant on each $S_\theta$. Since every Riemann surface $S_\theta$ contains the point $(0,0)$, this implies $f$
 must be constant on $b \hb^+$.

\end{proof}

  Note that in the case of the unbounded Hartogs triangle in $\cb^2$, the function $f$ needs to be of class $\ci$ on $\ol\hb^+$ to be extendable as a holomorphic function to $\cb^2$ (see Proposition \ref{smooth} and the beginning of the proof of Theorem \ref{ck}). But  in $\cx\mathbb P^2$, in contrary to $\cb^2$ we get (compare to Corollary \ref{vanishsmooth} and Theorem \ref{ck})  from the previous propositions that
\begin{cor}
For each $k\geq 0$, $H^{0,1}_{\ol\hb^-,\cc^k}(\cx\mathbb P^2)=0$ and $H^{0,1}_{\ol\hb^-,\infty}(\cx\mathbb P^2)=0$ .
\end{cor}

As in the case of $\cb^2$, we get for extendable  currents
\begin{prop}\label{curproj}
Suppose that  $H^{0,1}_{\ol\hb^-,cur}(\cx\mathbb P^2)=0$.  Then any
holomorphic function on $\hb^+$, which is extendable in the sense of
currents, is constant.
\end{prop}

\begin{thm}
 $H^{0,1}_{\ol\hb^-,cur}(\cx\mathbb P^2)$ does not vanish and  is Hausdorff.
\end{thm}
\begin{proof}
Let us consider the function $h$ defined on the open subset $\hb^+$ of $\cx\mathbb P^2$
by $h([z_0:z_1:z_2])=\frac{z_1}{z_2}$. It is holomorphic and bounded
and hence defines an extendable
current, but it is not constant, so by Proposition \ref{curproj}, we
get $H^{0,1}_{\ol\hb^-,cur}(\cx\mathbb P^2)\neq 0$.
  By the  Serre duality, to prove that $H^{0,1}_{\ol\hb^-,cur}(\cx\mathbb P^2)$ is Hausdorff, it is sufficient to prove that $H^{2,2}_\infty(\ol\hb^-)=0$.

  Let $f$ be a smooth $(2,2)$-form on $\ol\hb^-$ and $U$ be a neighborhood of $\ol\hb^-$, we can choose $U$ such that $\ol U$ is a connected proper subset of $\cx\mathbb P^2$. Then $f$ extends as a smooth $(2,2)$-form on $U$,  called $\wt f$. By Malgrange's theorem,  the top degree Dolbeault cohomology group $H^{2,2}(U)$ vanishes since $U$ is a non compact connected complex manifold. Thus  there exists a smooth $(2,1)$-form $u$ on $U$ such that $\opa u=\wt f$ on $U$. Then $v=u_{|_{\ol\hb^-}}$is a smooth form on $\ol\hb^-$ which satisfies $\opa v=f$ on $\hb^-$.
\end{proof}

Let us now consider the $L^2$ Dolbeault cohomology  with prescribed support  in an Hartogs triangle in $\cx\mathbb P^2$.
As usual we endow $\hb^+$ with the restriction of the  Fubini-Study metric of $\cx \mathbb{P}^2$.
  The following   proposition is already proved in Proposition 6 in  \cite{CS2012}.

 \begin{prop}\label{holoL2}
Let $\hb^+\subset\cx\mathbb P^2$ be the Hartogs' triangle. 
Then we have the following:
\begin{enumerate}
\item The Bergman space of $L^2$ holomorphic functions   $L^2(\hb^+)\cap\mathcal{O}(\hb^+)$
on the domain $\hb^+$ separates points in $\hb^+$.

\item There exist nonconstant  functions in the space $W^1(\hb^+)\cap\mathcal{O}(\hb^+)$.
However, this space does not separate points in $\hb^+$ and  is not dense in the Bergman space $L^2(\hb^+)\cap \mathcal{O}(\hb^+)$.
\item  Let $f\in W^2(\hb^+)\cap \mathcal{O}(\hb^+)$ be a holomorphic function on $\hb^+$ which is in the
Sobolev space $ W^2(\hb^+)$. Then $f$ is a constant.
  \end{enumerate}
\end{prop}

\begin{prop}\label{extW1}
Let $\hb^+\subset\cx\mathbb P^2$ be the Hartogs' triangle.   Any function $f\in  W^1(\hb^+)\cap\mathcal{O}(\hb^+)$ can be extended to a function in $W^1(\cx\mathbb{P}^2)$.
\end{prop}
\begin{proof} In the non-homogeneous holomorphic coordinates $(z,w)$ for $\Bbb H^+$,  any function $f\in  W^1(\hb^+)\cap\mathcal{O}(\hb^+)$ has the form (see Proposition 6 in \cite{CS2012})
$$f_k(z,w)= \left( \frac zw\right)^k,  \qquad k\in \mathbb N.$$
 It suffices to prove the proposition for each $f_k(z,w)$.

Let $\chi(t)\in C^\infty(\mathbb R) $ be a function defined by $\chi(t)=0$ if $t\le 0$ and $\chi(t)=1$ if $t\ge 1$.
Let $\tilde f_k$ be the function defined by
\begin{equation} \tilde f_k(z,w)= \chi\left(1+\frac 13(1- \frac {|z|^2 }{|w|^2 })\right)f_k (z,w).\end{equation}
On  $|z|<|w|$, it is easy to see that $\tilde f_k=f_k$. Thus $\tilde f_k$ is an extension of $f_k$ to $\cx\mathbb{P}^2$.

To see that $\tilde f_k$ is in $W^1(\cx\mathbb{P}^2)$,  we first note that the function
 $$\chi\left(1+\frac 13(1- \frac {|z|^2 }{|w|^2 })\right)=0$$ when restricted to $\{|z|\ge  2|w|\}$.
  Thus it is   supported in $\{|z|\le 2|w|\}$. On its support, the function $\frac {|z|}{|w|}$ is bounded. Using this fact and  differentiating under the  chain rule, we have  that
\begin{equation}\label{deriv}
|\nabla \chi\left(1+\frac 13(1- \frac {|z|^2 }{|w|^2 })\right)|\le  C( \sup|\chi'|) \frac 1{|w|}\le C\frac 1{|w|}.\end{equation}
 Repeating the arguments as before,  we see that the function $\frac 1{|w|}$ is in  $L^2$ on $\{|z|\le 2|w|\}$.
   Since the function $f_k$ is bounded on the set $\{|z|\le 2|w|\}$,  we conclude from \eqref{deriv}  that the derivatives of  $\tilde f_k$ is in   $L^2(\cx\mathbb{P}^2)$. Thus $\tilde  f_k$ is an extension in $W^1(\cx\mathbb{P}^2)$ of $f_k$.

 \end{proof}

 {\parindent=0pt{\bf Remark.} Suppose $D$ is a bounded  domain  with Lipschitz boundary, then any function $f\in W^1(D)$ extends as a
 function in $W^1(\cx\mathbb P^2)$. It is not known if this is true for the Hartogs triangle $\hb^+$. In the proof of Proposition \ref{extW1}, we have used the fact that the function  $f_k$ are in $W^1(\hb^+)$ and  {\it bounded} on $\hb^+$.}

 \begin{thm}\label{infinite}
Let $\hb^-\subset\cx\mathbb P^2$ be the Hartogs' triangle.   Then the cohomology group
 $H^{0,1}_{\ol\hb^-, L^2}(\cx\mathbb P^2)\neq 0$  and is infinite dimensional.
 \end{thm}
 \begin{proof} We recall that
  $\hb^+= \cx\mathbb{P}^2\setminus \overline \hb^-.$
 From Proposition \ref{holoL2},  the space of holomorphic functions in $W^1(\hb^+)\cap\mathcal{O}(\hb^+)$ is infinite dimensional.  In the non-homogeneous coordinates, consider the holomorphic functions of the type $f_k=( \frac zw)^k$, $k\in \mathbb N$.

We define the operator $\opa_{\tilde c}$ as the weak minimal realization of $\opa$, then the domain of $\opa_{\tilde c}$ is the space of $L^2$ forms $f$ in $\cx\mathbb P^2$ with support in $\ol\hb^-$ such that $\opa f$ is also an $L^2$ form in $\cx\mathbb P^2$.

 Using Proposition \ref{extW1}, each holomorphic function $f_k$ can be extended to a function $\tilde f_k\in W^1(\cx \mathbb P^2)$.
 Suppose that  $H^{0,1}_{\ol\hb^-, L^2}(\cx\mathbb P^2)=0$. Then we can solve
  $\bar \partial_{\tilde c} u_k= \dbar \tilde f_k$
 in $\cx\mathbb P^2$ with prescribed support for $u_k$  in $\overline \hb^-$.  Let
  $H_k=\tilde f_k-u_k.$
 Then $H_k$ is a  holomorphic function in $\cx\mathbb P^2$, hence  a constant. But $H_k=f_k$ on $\hb^+$, a contradiction.
 This implies that the space $H^{0,1}_{\ol\hb^-, L^2}(\cx\mathbb P^2)$  is non-trivial.

 Next we prove that  $H^{0,1}_{\ol\hb^-, L^2}(\cx\mathbb P^2)$ is infinite dimensional.     Each function $\tilde f_k$ corresponds to a
 (0,1)-form
  $ \dbar \tilde f_k$. We set  $g_k=\dbar \tilde f_k$. Then  $g_k$  is in $\text{Dom}(\dbar_{\tilde c})$ and satisfies
   $\dbar_{\tilde c} g_k=0.$  Thus it  induces an element  $[g_k]$  in $H^{0,1}_{\ol\hb^-, L^2}(\cx\mathbb P^2)$. To see that $[g_k]$'s are linearly independent, let $N>1$ be a positive integer  and  $F_N=\sum_{k=1}^{N} c_k  f_k$, where $c_k$ are constants.  Set
  $G_N=\sum_{k=1}^{N} c_k  g_k.$
  Suppose that
  $[G_N]=0$, then we can solve $\dbar_{\tilde c} u= G_N$ and    the  function $F_N$ holomorphic in $\hb^+$ extends holomorphically to $\cx\mathbb P^2$.  Thus $F_N$ must be a constant and
   $c_1=\cdots=c_N=0.$  Thus $[g_k]$'s are linearly independent.
This proves that   $H^{0,1}_{\ol\hb^-, L^2}(\cx\mathbb P^2)$ is infinite dimensional.
 \end{proof}

{\parindent=0pt{\bf Remark.} It follows from Proposition \ref{injective} and Theorem \ref{infinite} that $H^{0,1}_{\ol\hb^-, cur}(\cx\mathbb P^2)$ is also infinite dimensional.}

 \begin{lem}\label{range}
 The  range of the strong $L^2$ closure of $\opa$
 \begin{equation}\dbar_s:L^2_{2,1}(\hb^-)\to L^2_{2,2}(\hb^-) \end{equation}
  is closed and  equal to $L^2_{2,2}(\hb^-)$.
 \end{lem}
  \begin{proof}
It is clear  that $\dbar$ has closed range in the top degree and the range is $L^2_{2,2}(\hb^-)$.   Let $f\in L^2_{2,2}(\hb^-)$.
 We extend $f$ to be zero outside $ \hb^-$.
Let $U$ be an open neighbourhood of $\overline \hb^-$, then $f$ is in $L^2_{2,2}(U)$. We can choose $U$ such that $\overline U$ is a proper subset of $\cx\mathbb P^2$  and $U$ has  Lipschitz boundary.   Since one can solve the $\dbar$ equation for  top degree forms  on $U$, there exists       $u\in L^2_{2,1}(U)$ such that
  $$\dbar u=f$$
  in the weak sense.

   It suffices to show that $f$ is in the range of $\dbar_s$. Since $U$ has Lipschitz boundary,
using Friedrichs' lemma, there exists a sequence $u_\nu\in C^\infty(\overline U)$ such that $u_\nu\to u$  and $\dbar u_\nu\to f$ in $L^2_{2,2}(U)$.  Restricting $u_\nu$ to $\overline \hb^-$, we have that   $u$ is  in the domain of $\dbar_s$ and
$$\dbar_s u=f.$$ Thus the range of
$\dbar_s$ is equal to $L^2_{2,2}(\hb^-)$.
   The lemma is proved.
  \end{proof}

  \begin{cor}\label{hausdorff} The  cohomology group
 $H^{0,1}_{\ol\hb^-, L^2}(\cx\mathbb P^2)$   is Hausdorff and infinite dimensional.
 \end{cor}

  \begin{thm} Let us consider the Hartogs'  triangle $\hb^-\subset\cx\mathbb{P}^2$.
   Then the cohomology group
 $H^{2,1}_{\dbar_s, L^2}(\hb^-)$ is infinite dimensional.
 \end{thm}

 \begin{proof}
    Suppose that  $\dbar_s:L^2_{2,0}(\hb^-)\to L^2_{2,1}(\hb^-)$ does not have  closed range.
Then       $H^{2,1}_{\dbar_s, L^2}(\hb^-) $
is non-hausdorff, hence   infinite dimensional.

Suppose that  $\dbar_s:L^2_{2,0}(\hb^-)\to L^2_{2,1}(\hb^-)$  has closed range.
 Using Lemma \ref{range},  $\dbar_s:L^2_{2,1}(\hb^-)\to L^2_{2,2}(\hb^-)$ has closed range.   From the $L^2$ Serre  duality,
     $\dbar_{\tilde c} :L^2(\hb^-) \to L^2_{0,1} (\hb^-)$ and  $\dbar_{\tilde c} :L^2_{0,1}(\hb^-) \to L^2_{0,2} (\hb^-)$ both have closed range. Furthermore,
  \begin{equation} H^{2,1}_{\dbar_s, L^2}(\hb^-) \cong H^{0,1}_{\ol\hb^-, L^2}(\cx\mathbb P^2).\end{equation}
  Thus from Theorem \ref{infinite},  it is infinite dimensional.

 \end{proof}

\bigskip
\noindent{\bf Remarks:}

\begin{enumerate}

 \item Let $\mathbb T=\{(z_1,z_2)\in \cx^2\mid |z_2|< |z_1|<1\}$ be the Hartogs triangle in $\cx^2$. Then by Proposition \ref{CNconnected},
 $$H^{0,1}_{\dbar_{\tilde c}, L^2}(\mathbb T) =H^{0,1}_{ \overline {\mathbb T}, L^2}(\cx^2)=0.$$
 This is in sharp contrast to Corollary \ref{hausdorff}.

It is well-known that $H^{01}(\mathbb T)=0$ since $\mathbb T$ is pseudoconvex,  but  $H^{0,1}_\infty(\overline{\mathbb T})$ (cohomology with forms smooth up to the boundary) is infinite dimensional (see \cite{Si}). In fact, $H^{0,1}(\overline{\mathbb T})$ is even   non-Hausdorff (see
     \cite{LaSh2}). We also refer the reader to the recent survey paper on the Hartogs triangle \cite{Sh}.

\item  If  $D$ is  a domain in $\cx\mathbb P^n$ with  $C^2$ boundary,  then we have  $L^2$ existence theorems  for $\dbar$ on $D$ for all degrees (see \cite{BC}  \cite{HI},   \cite{CSW}).  This follows from the existence of bounded plurisubharmonic functions on pseudoconvex domains in $\cx\mathbb P^n$ with $C^2$ boundary (see \cite{OS}).  This is even true if $D$ has only Lipschitz boundary (see \cite{Ha}).

\item Suppose that $D$ is a pseudoconvex domain  in $\cx\mathbb P^n$ with Lipschitz boundary, we have $H^{p,q}_{L^2}(D)=0$ for all $q>0$.
By the $L^2$ Serre duality (see \cite{CS2012}), we have
  $H^{0,1}_{\dbar_c, L^2} (D)= H^{0,1}_{\overline D, L^2}(\cx\mathbb P^n)=0.$
 Corollary \ref{hausdorff} shows that the Lipschtz  condition cannot be removed.

\item From a result of Takeuchi \cite{Ta}, $\hb^-$ is Stein.   It is well-known that  for any $p$, $0\le p\le 2$,
  $\dbar :L^2_{p,0}(\hb^-,\text{loc}) \to L^2_{p,1}(\hb^-,\text{loc})$ has closed range (see   \cite{Ho}) and the cohomology
  $H^{p,1}_{L^2_{\text{loc}}}(\hb^-)$  in
the Frech\'et space   $L^2_{0,1}(\hb^-,\text{loc})$ is trivial.

 \item   The (weak) $L^2$ theory   holds for any pseudoconvex  domain without any regularity assumption on the boundary for $(0,1)$-forms.
  The (weak)  $L^2$ Cauchy-Riemann operator
 $\dbar:L^2(\hb^-)\to L^2_{0,1}(\hb^-)$   has closed range and  $H^{0,1}_{L^2}(\hb^-)=0$ (see  \cite{HI} or \cite{CSW}).

\item For $p=1$ or $p=2$, it is not known  if
  the Cauchy-Riemann operator  $\dbar:L^2_{p,0}(\hb^-)\to L^2_{p,1}(\hb^-)$   has closed range. It is also not known if    $\dbar$ in the weak sense   is  equal to  $\dbar_s$.

\item It is not known if the   strong $L^2$ Cauchy-Riemann operator $\dbar_s:L^2_{2,0}(\hb^-)\to L^2_{2,1}(\hb^-)$   has     closed range.

\end{enumerate}

\providecommand{\bysame}{\leavevmode\hbox to3em{\hrulefill}\thinspace}
\providecommand{\MR}{\relax\ifhmode\unskip\space\fi MR }
\providecommand{\MRhref}[2]{%
  \href{http://www.ams.org/mathscinet-getitem?mr=#1}{#2}
}
\providecommand{\href}[2]{#2}

\enddocument

\end